\documentclass[11pt]{amsart}
\calclayout
\usepackage{amsmath,amsfonts,amssymb,amsthm,epsfig,color}
\usepackage{graphicx}
\usepackage{subcaption}
\usepackage[english]{babel}
\usepackage[T1]{fontenc}
\usepackage{caption}
\usepackage{verbatim} 
\usepackage{enumerate}
\usepackage{tikz}
\usetikzlibrary{decorations.markings}
\usetikzlibrary{shapes}
\usepackage{float}
\usepackage{bbm}
\usepackage[hidelinks]{hyperref}
\usepackage[nameinlink,capitalize,noabbrev]{cleveref}

\newtheorem*{thm*}{Theorem}

\newtheorem{thm}{Theorem}[section]
\newtheorem{cor}[thm]{Corollary}

\newtheorem{lemma}[thm]{Lemma}
\newtheorem{prop}[thm]{Proposition}

\newtheorem{rem}[thm]{Remark}

\theoremstyle{definition}
\newtheorem{defn}[thm]{Definition}

\newcommand{\R}{\mathbb{R}}

\newcommand{\N}{\mathbb{N}}
\newcommand{\Z}{\mathbb{Z}}
\newcommand{\C}{\mathbb{C}}

\newcommand{\Id}{\mathrm{Id}}

\newcommand{\ud}{\mathrm{d}}

\title{Invariant sets for the wind-tree model}
\author{Yuriy Tumarkin}
\address{Institut f\"ur Mathematik, Universit\"at Z\"urich, Winterthurerstrasse 190, CH-8057 Z\"urich, Switzerland}
\email{yuriy.tumarkin@math.uzh.ch}

\begin{document}
\maketitle
\begin{abstract}
We consider the wind-tree model, a $\Z^2$-periodic billiard. In the case when the underlying compact translation surface lies on a periodic orbit of the Teichmüller geodesic flow, and at least one of the two homology classes defining the cover is unstable for the Kontsevich-Zorich cocycle, we prove that every orbit closure of the billiard has Hausdorff dimension strictly smaller than 2. The proof relies on a construction of explicit invariant functions, which along the way gives a new proof of non-ergodicity and non-transitivity of the wind-tree model for all parameters and almost all directions, as first shown by Fr\c{a}czek and Ulcigrai in \cite{FU}.

\end{abstract}

\section{Introduction}

The Ehrenfest wind-tree model is a billiard in the plane with rectangular obstacles. First introduced by P. Ehrenfest and T. Ehrenfest in 1911 \cite{Ehrenfest} to model gas diffusion, the periodic version as first studied by Hardy and Weber \cite{HardyWeber} has attracted a lot of attention in the last decade and a half. It is defined as follows: a point (the `wind') moves along straight lines in the plane $\R^2$ and bounces elastically off rectangular obstacles (the `trees'), which are translates of the rectangle $[0,a] \times [0,b]$ for some parameters $0<a,b<1$, centered at each element of $\Z^2$.

Most of the recent progress on the wind-tree model (\cite{HLT,CG,Delecroix,DHL,FU,FH,DZ,AH,CHLP}) has been due to the following observation. Following the classical unfolding construction of Katok and Zemlyakov \cite{KatokZemlyakov}, the billiard flow on the wind-tree model is equivalent to the straight-line flow on an infinite translation surface $X_\infty$ (see \Cref{sec:background} for definitions), which turns out to be a $\Z^2$-cover of a genus 5 compact translation surface $X$, which in turn is a fourfold cover of a genus 2 translation surface $Y$. This reduction to compact translation surfaces allows one to apply deep results from Teichmüller dynamics (such as the formula of Eskin, Kontsevich and Zorich \cite{EKZ} for the Lyapunov exponents of the Kontsevich-Zorich cocycle and the work of Chaika and Eskin \cite{ChaikaEskin} on Oseledets genericity).

Among the most celebrated results on the wind-tree model is the result of Fr\c{a}czek and Ulcigrai \cite{FU}, showing that for any parameters $a,b$, the billiard flow in almost any direction $\theta$ is non-ergodic, and further even non-transitive, meaning that it has no dense orbits. This is in stark contrast to well-known results about compact translation surfaces, where by the work of Kerckhoff, Masur and Smillie \cite{KMS}, for any compact translation surface the straight-line flow is ergodic (and even uniquely ergodic) in almost every direction.

The work of Fr\c{a}czek and Ulcigrai opens many natural questions which need to be answered to obtain a full understanding of the dynamics of the wind-tree model: 
\begin{itemize}
\item We know that Lebesgue measure is not ergodic for the flow, so what are its ergodic components? 
\item Further, we know that generically no orbits are dense, so what are their closures? 
\end{itemize}
The goal of this article is to take a first step towards answering these questions, with a hope of spurring further progress in this direction.

Denote by $X(a,b)$ the compact translation surface underlying the wind-tree model with parameters $(a,b)$, and by $r_{\pi/2-\theta}$ the rotation by $\pi/2-\theta$ (see \Cref{sec:background} for precise definitions). We discuss the case when $r_{\pi/2-\theta} X(a,b)$ is periodic for the Teichmüller geodesic flow. The cover $r_{\pi/2-\theta} X_\infty(a,b)$ is defined by two homology classes $\gamma_h, \gamma_v \in H_1(X,\Z)$ (see \Cref{sec:covers}).

The main new result that we obtain is as follows:

\begin{thm}\label{thm:main}
Let $a,b \in (0,1)$ be any parameters, let $\theta$ be a direction such that $r_{\pi/2-\theta} X(a,b)$ is periodic for the Teichmüller geodesic flow and at least one of $\gamma_h$ and $\gamma_v$ is unstable. Then for the billiard flow on the wind-tree model with parameters $a,b$ in direction $\theta$, the closure of every orbit has Hausdorff dimension strictly less than 2.
\end{thm}

\begin{rem}
By \emph{unstable} we mean that $\gamma_h$ or $\gamma_v$ belongs to the the unstable space $E^+_\omega$. (See \Cref{sec:background}). 

Since $\gamma_h$ and $\gamma_v$ are integer vectors, they can never be stable, but they can in some cases be central. \\

Let $\Psi$ be the pseudo-Anosov diffeomorphism of $X(a,b)$ corresponding to the closed Teichmüller geodesic containing $r_{\pi/2-\theta} X(a,b)$. If both $\gamma_h$ and $\gamma_v$ are invariant under the action of $\Psi$, this corresponds to the case when $\Psi$ lifts to the cover $X_\infty(a,b)$. This case was studied in \cite{Tum}, where it was shown that in this case the flow is ergodic, and further one has a full classification of Radon ergodic invariant measures as the Maharam measures.

As far as we are aware, other central cases, when $\Psi$ does not lift, are still open.
\end{rem}

Fraçzek and Hubert extended the results of \cite{FU} to more general translation surfaces and higher dimensional covers in \cite{FH}, and in fact our result holds in this more general setting.

The setup and large part of the technical input for our proof is along the same lines as that of \cite{FU, FH}. We also use the observation that the homology of $X$ splits into several invariant blocks for the Kontsevich-Zorich cocycle, and that the homology classes defining the cover $X_\infty$ are complemented by homology classes which are stable for the cocycle. These stable classes determine cocycles over the first-return map of the flow which are coboundaries. From here, while \cite{FU} and \cite{FH} use the theory of essential values to conclude non-ergodicity, we rely on a result of Marmi, Moussa and Yoccoz \cite{MMY1} which states that these coboundaries have continuous transfer functions. This allows us to find continuous invariant functions for the flow, which provide the main ingredient in the proof of \Cref{thm:main}.\\

Another application of the explicit form of the invariant functions is that one can use them to plot computational approximations of invariant sets for the billiard flow.
An example with $\theta$ a $g_t$-periodic direction is shown in \Cref{fig:InvSet}.

\begin{figure}[h]
\centering
\includegraphics[width=0.8\textwidth]{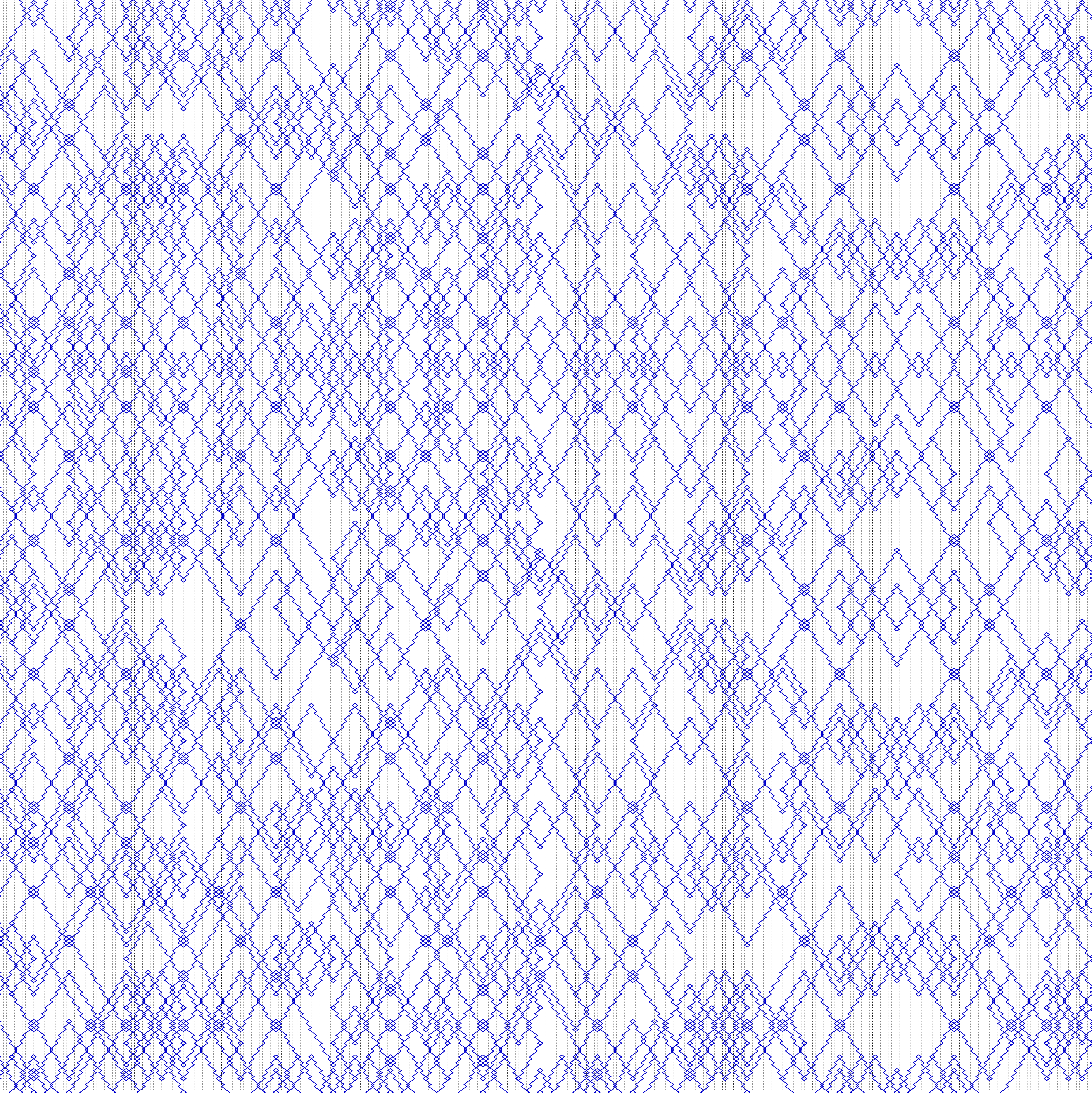}
\caption{A $400\times 400$ window of an invariant set for the billiard in the wind-tree model, in the $g_t$-periodic case.}
\label{fig:InvSet}
\end{figure}

The outline of the paper is as follows. In \Cref{sec:background} we introduce the necessary background on translation surfaces and interval exchange transformations. In \Cref{sec:covers} we discuss covers of translation surfaces and the geometry of the wind-tree model and the relevant translation surfaces. In \Cref{sec:invariant} we prove the existence of continuous invariant functions for the straight-line flow. Finally in \Cref{sec:hdim} we show that in the unstable $g_t$-periodic case the invariant sets given by level sets of the aforementioned invariant functions have Hausdorff dimension less than 2, proving \Cref{thm:main}. In addition, in \Cref{sec:plotting} we say a few words on how one can plot the invariant sets as in \Cref{fig:InvSet}.\\

\subsection{Acknowledgments}
I would like to thank my advisor Corinna Ulcigrai for teaching me about the Ehrenfest wind-tree model and encouraging me to work on this problem, and for her continuous support during this work. I am grateful to Krzystof Fr\c{a}czek, Pascal Hubert and Anton Zorich for helpful discussions, and I would like to thank Andrea Ulliana for showing me how to simplify the proof of \Cref{thm:hdim}. This work was supported by the Swiss National Science Foundation through Grant 213663.

\section{Background on translation surfaces and IETs}\label{sec:background}

A \emph{translation surface} is a tuple $(M,\omega)$, where $M$ is an oriented surface (which may not be compact) and $\omega$ is a complex structure on $M$ with an \emph{abelian differential}, i.e. a non-zero 1-form which is holomorphic with respect to the complex structure. We denote by $\Sigma \subset M$ the set of zeros of $\omega$, also called the \emph{singular points} of $(M,\omega)$. 

For a direction $\theta \in S^1$, denote by $X^\theta$ the unit vector field in direction $\theta$ on $M\setminus \Sigma$. Denote by $(\varphi^\theta_t)_{t\in \R}$ the corresponding flow, which is called the \emph{straight-line flow in direction $\theta$}. We denote the straight-line flow in the vertical direction by $\varphi^v_t$. The flow $(\varphi^\theta_t)_{t\in \R}$ preserves the natural volume form on $(M,\omega)$ given by $\Re \omega \wedge \Im \omega$.

\subsection{Teichmüller flow}
\subsubsection{Teichmüller space and moduli space}
Denote by $\mathrm{Diff}^+(M)$ the group of orientation-preserving diffeomorphisms of $M$, and by $\mathrm{Diff}_0^+(M)$ the subgroup of diffeomorphisms isotopic to the identity. The quotient $\Gamma(M) := \mathrm{Diff}^+(M)/\mathrm{Diff}^+_0(M)$ is the \emph{mapping class group} of $M$. 

The group $\mathrm{Diff}^+(M)$ acts on the space of abelian differentials on $M$ by pre-composition. Denote by $\mathcal{T}(M)$ the \emph{Teichmüller space of abelian differentials}, the space of orbits of the action restricted to $\mathrm{Diff}_0^+(M)$. Denote by $\mathcal{M}(M)$ the \emph{moduli space of abelian differentials}, the space of orbits of the action of the entire group $\mathrm{Diff}^+(M)$. Then $\mathcal{M}(M) = \mathcal{T}(M)/\Gamma(M)$.\\

Defining the \emph{area} of an abelian differential $\omega$ by $A(\omega)  = \int_M \Re \omega \wedge \Im \omega$, we define also the Teichmüller and moduli spaces of unit-area abelian differentials by $\mathcal{T}_1(M) := \{[\omega] \in \mathcal{T}(M): A(\omega) = 1\},\, \mathcal{M}_1(M) := \{[\omega] \in \mathcal{M}(M): A(\omega) = 1\}$.

\subsubsection{The $SL(2,\R)$ action}
The group $SL(2,\R)$ has a natural action on $\mathcal{T}(M)$ and $\mathcal{M}(M)$ in the following way. Given a translation structure $\omega$, its primitives define charts to $\C$. Post-composition of these charts by an element $g$ of $SL(2,\R)$ gives new charts that define a new complex structure and a new holomorphic abelian differential, thus defining a new translation structure $g\cdot \omega$.

The \emph{Teichmüller geodesic flow} $(g_t)_{t\in \R}$ is the restriction of the $SL(2,\R)$ action to the diagonal subgroup 
$\{g_t := \mathrm{diag}(e^{t},e^{-t}) : t\in \R\}$.

Say that $\omega$ is \emph{$g_t$-periodic} if it is periodic under the Teichmüller geodesic flow on the moduli space, meaning that for some $T_0$, $g_{T_0} \cdot \omega = \omega$ in $\mathcal{M}(M)$.

We will also sometimes refer to the restriction of the action to $SO(2,\R)$, denoting by $r_\theta$ the matrix of rotation by $\theta$.

\subsubsection{The Kontsevich-Zorich cocycle}
Define the \emph{homological Hodge bundle} over $\mathcal{M}(M)$ by 
\[\mathcal{H}_1(M) :=  (\mathcal{T}_1(M) \times H_1(M,\R)) / \Gamma(M).\]
Here $\Gamma(M)$ acts on $\mathcal{T}_1(M)$ by pre-composition and on $H_1(M,\R)$ by the natural pushforward action. \\

The \emph{Kontsevich-Zorich cocycle} $(G^{KZ}_t)_{t\in \R}$ is the quotient by $\Gamma(M)$ of the trivial cocycle 
\[g_t \times \Id : \mathcal{T}_1(M) \times H_1(M,\R) \to \mathcal{T}_1(M) \times H_1(M,\R)\]
to $\mathcal{H}_1(M,\R)$.

Note that the Kontsevich-Zorich cocycle is usually defined on the cohomological Hodge bundle $\mathcal{H}^1(M) :=  (\mathcal{T}_1(M) \times H^1(M,\R)) / \Gamma(M)$, but we can identify the homological and cohomological bundles via the Poincaré duality $\mathcal{P}: H_1(M,\R) \to H^1(M,\R)$. This identification allows us to define the Hodge norm on the homological bundle by pulling back the Hodge norm on the cohomological bundle. 

The \emph{Hodge norm} on $\mathcal{H}^1(M,\R)$ is defined as follows. Given $ c \in H^1(M,\R)$, there exists a unique one-form $\eta$ on $M$ which is holomorphic with respect to the complex structure of $\omega$ and satisfies $c = [\Re \eta] \in H^1(M,R)$. Then the Hodge norm of $c$ is given by 
\[\|c\|_\omega := \left(\frac{1}{2}\int_M \eta \wedge \bar \eta\right)^\frac{1}{2}.\]

We will denote the Hodge norm on $\mathcal{H}_1(M,\R)$ by $\|\sigma\|_\omega := \|\mathcal{P}(\sigma)\|_\omega$.\\

For a compact orientable surface $M$, the homology $H_1(M,\R)$ has a natural symplectic structure arising from the algebraic intersection number. This symplectic structure is preserved by the action of $\Gamma(M)$ and hence the Kontsevich-Zorich cocycle acts symplectically.

\subsubsection{Orbit closures and $SL(2,\R)$-invariant subbundles}
Let $\omega \in \mathcal{M}_1(M)$ and let $\mathcal{N} = \overline{SL(2,\R)\cdot\omega}$ be its orbit closure. By the work of Eskin, Mirzakhani and Mohammadi in \cite{EM,EMM}, $\mathcal{N}$ is an affine $SL(2,\R)$-invariant submanifold of $\mathcal{M}_1(M)$, and there is a probability measure $\nu_\mathcal{M}$ supported on $\mathcal{N}$ which is invariant for the $SL(2,\R)$ action and ergodic for the Teichmüller geodesic flow.

Suppose that $V \subset H_1(M,\R)$ is a symplectic subspace.
An \emph{$SL(2,\R)$-invariant subbundle over $\mathcal{N}$} with fibre $V$ is a subbundle $\mathcal{V} = \bigcup_{\omega \in \mathcal{N}} \{\omega\} \times V(\omega)$ of $\mathcal{H}_1(M,\R)$, such that for all $\omega \in \mathcal{N}$, $V(\omega) \cong V$, and for any $g\in SL(2,\R)$, $V(g\cdot \omega) = V(\omega)$.

We denote the Kontsevich-Zorich cocycle restricted to $\mathcal{V}$ by $(G^\mathcal{V}_t)_{t\in\R}$. 

\subsection{Generic properties of a translation surface}
Let $\mathcal{N} = \overline{SL(2,\R)\cdot \omega}$, and let $\nu_\mathcal{N}$ be the corresponding $SL(2,\R)$-invariant measure on $\mathcal{N}$. Fix a symplectic subspace $V \subset H_1(M,\R)$ and consider an invariant subbundle $\mathcal{V}$ over $\mathcal{N}$ with fibre $V$.

We state three important properties satisfied by a $\nu_\mathcal{N}$-generic translation surface $(M,\omega) \in \mathcal{N}$. We first state the definitions of the properties, then the theorems which say that these properties are indeed generic. In fact stronger than holding for $\nu_\mathcal{N}$- almost every $(M,\omega)$, by the results of Chaika-Eskin \cite{ChaikaEskin}, these properties hold for every surface when rotated by almost every $\theta \in S^1$.

\begin{defn} We say that $(M,\omega)$ is \emph{Birkhoff generic} if the Birkhoff ergodic theorem holds for the Teichmüller geodesic flow at $\omega$, i.e. for every continuous integrable function $\phi$ on $\mathcal{N}$, 
\[\lim_{T \to \infty} \int_0^T \phi(g_t\omega) \,\ud t = \int_\mathcal{N} \phi \,\ud \nu_\mathcal{N}.\]
\end{defn}

\begin{defn} We say that $(M,\omega)$ is \emph{Oseledets generic} if there exists an Oseledets splitting for the Kontsevich-Zorich cocycle $G^\mathcal{V}_t$ with respect to $\nu_\mathcal{N}$, i.e. there exist Lyapunov exponents 
\[\lambda_1 > \lambda_2 > \dots > \lambda_s,\]
and a direct splitting $V = \bigoplus_{i=1}^s V_i(\omega)$, such that or all $\xi \in V_i$,
\[\lim_{t\to \infty} \frac{1}{t} \log \|\xi\|_{g_t\omega} = \lambda_i.\]
\end{defn}  

If $\omega$ is Oseledets generic, $V$ has a direct splitting into \emph{stable}, \emph{central} and \emph{unstable} subspaces,
\[V = E^-_\omega \oplus E^0_\omega \oplus E^+_\omega,\]
where
\begin{align*}
E^-_\omega := \{\xi \in V: \lim_{t\to\infty} \frac{1}{t} \log \|\xi\|_{g_t\omega} < 0\},\\
E^0_\omega := \{\xi \in V: \lim_{t\to\infty} \frac{1}{t} \log \|\xi\|_{g_t\omega} = 0\},\\
E^+_\omega := \{\xi \in V: \lim_{t\to\infty} \frac{1}{t} \log \|\xi\|_{g_t\omega} > 0\}.
\end{align*}

\begin{defn}
We say that $(M,\omega)$ is \emph{Masur generic} if both the vertical and the horizontal flows are uniquely ergodic on $(M,\omega)$.
\end{defn}

Say that $(M,\omega)$ is \emph{BOM generic} if it is generic in all three senses above.

\begin{thm}[Theorems 1.1 and 1.2 in \cite{ChaikaEskin}]
\hspace{1cm}

\noindent For every $(M,\omega)$, and almost every $\theta \in S^1$, $(M,r_\theta \omega)$ is Birkhoff and Oseledets generic.
\end{thm}

\begin{thm}[Theorem 2 in \cite{KMS}]
\hspace{1cm}

\noindent For every $(M,\omega)$, and almost every $\theta \in S^1$, $(M,r_\theta \omega)$ is Masur generic.
\end{thm}

\subsection{IETs}
Fix a transverse segment $I$ to the direction $\theta$ on a translation surface. Since  the straight-line flow $(\varphi^\theta_t)_{t\in \R}$ is volume preserving, its Poincaré map on $I$ is an orientation-preserving piecewise isometry of the interval. Such maps are called \emph{interval exchange transformations}.

An \emph{Interval Exchange Transformation} (IET for short) of $n$ intervals is a map $T:[a,b)\to [a,b) $ 
which is a piecewise 
orientation-preserving isometry, with all discontinuities contained in a set of $n-1$ marked points $\{a < x_1 <\dots < x_{n-1} < b\}$. 
Let $x_0 = a, x_n = b$ and let $I_i$ be the interval $[x_{i-1},x_i)$ for $1\leq i\leq n$. Denote the intervals of continuity of $T$ by $\{I_j: j\in \mathcal{A}\}$, where $\mathcal{A} = \{1,\dots,n\}$.

Say that $T$ satisfies \emph{Keane's condition} if for any $0\leq i \leq n$, any $m\geq 1$, $T^m(x_i)$ does not belong to the set of marked points 
$\{x_1,\dots,x_{n-1}\}$. Let $\mathcal{X}_n$ be the space of IETs of $n$ intervals satisfying Keane's condition. Let $\mathcal{X}^0_n \subset \mathcal{X}_n$ be 
the subspace of IETs defined on the unit interval $I = [0,1)$.

A key tool for studying IETs is the Rauzy-Veech renormalisation $\mathcal{R}$. 
We do not give an explicit definition of Rauzy-Veech renormalisation here, only recalling some basic properties, and refer the reader to the many excellent 
introductions available in the literature, such as \cite{Viana}, \cite{Yoccoz}, \cite{Zorich}.

The \emph{Rauzy-Veech induction} is a map $\hat{\mathcal{R}}:\mathcal{X}_n \to \mathcal{X}_n$. Given $T:[a,b)\to [a,b)$,   
the Rauzy-Veech induction procedure
produces a sequence of shrinking nested intervals $[a,b)=I^{(0)} \supset I^{(1)} \supset I^{(2)} \supset \dots$, which all have the same left endpoint. 
The map $\hat{\mathcal{R}}^k(T)$ is defined as the induced map of $T$ on $I^{(k)}$, also denoted by $T^{(k)}:I^{(k)}\to I^{(k)}$. 
The subintervals of $I^{(k)}$ permuted by $T^{(k)}$ are labelled by $I^{(k)}_i$ for $i \in \mathcal{A}$, where the order of these subintervals depends on $T$.

The \emph{Rauzy-Veech renormalisation} is a map $\mathcal{R}:\mathcal{X}^0_n \to \mathcal{X}^0_n$, defined by 
$\mathcal{R}(T) = \frac{\1T}{|\1I|}$.

For $0\leq r < k$ let $q^{(r,k)}_j$ be the return time of $T^{(r)}$ to $I^{(k)}$ on $I^{(k)}_j$, i.e. 
\[q^{(r,k)}_j = \min\Big\{m \geq 1 : \big(T^{(r)}\big)^m \Big( I^{(k)}_j\Big) \subset I^{(k)}\Big\}.\]

The \emph{Rokhlin tower} $Z^{(r,k)}_j$ is the union 
\[Z^{(r,k)}_j = \bigsqcup_{l=0}^{q^{(r,k)}_j-1} \big(T^{(r)}\big)^l \left( I^{(k)}_j \right).\]

Then by definition of the induced map and $q^{(r,k)}_j$, $I^{(r)} = \bigsqcup_{j=1}^n Z^{(r,k)}_j$.

For $0 \leq l < q^{(r,k)}_j$ we call $\big(T^{(r)}\big)^l \left( I^{(k)}_j \right)$ the $l^\text{th}$ \emph{floor} of the tower $Z^{(r,k)}_j$, and $I^{(k)}_j$ the \emph{base}.\\

The \emph{Rauzy-Veech cocycle} is defined as integer-valued $d\times d$ matrices $A^{(r,k)}$, with
\[A^{(r,k)}_{ij} = \#\left\{0\leq l < q^{(r,k)}_j: \big(T^{(r)}\big)^l \left(I^{(k)}_j\right) \subset I^{(r)}_i\right\}.\]
Thus $A^{(r,k)}_{ij}$ counts the number of floors in the tower $Z^{(r,k)}_j$ which are subsets of $I^{(r)}_i$.
It is a cocycle since for $r<s<k$, $A^{(r,k)}_{ij} = \sum_{t=1}^n A^{(r,s)}_{it} A^{(s,k)}_{tj}$, and so $A^{(r,k)} = A^{(r,s)}A^{(s,k)}$.\\

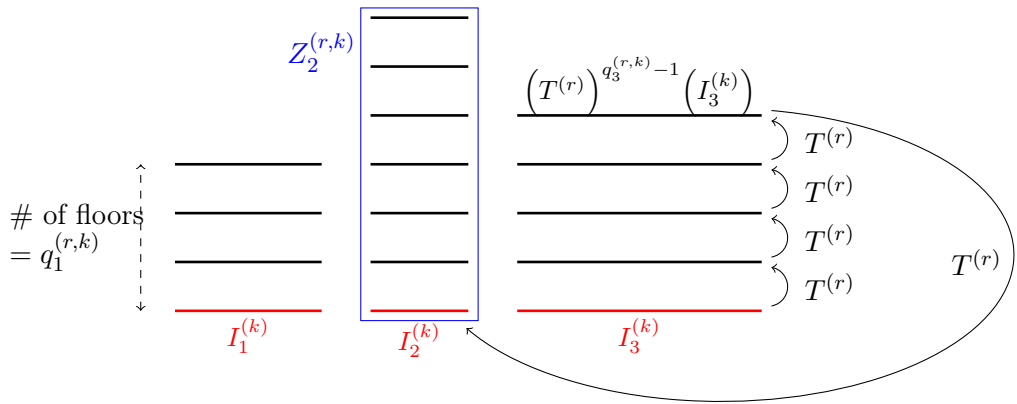
\begin{figure}[H]
\centering
\begin{tikzpicture}[scale=6.5]
	\draw[red,line width=1] (0,0) -- (0.3,0);
	\draw[red,line width=1] (0.4,0) -- (0.6,0);
	\draw[red,line width=1] (0.7,0) -- (1.2,0);
	
	\draw[line width=1] (0,0.1) -- (0.3,0.1);
	\draw[line width=1] (0.4,0.1) -- (0.6,0.1);
	\draw[line width=1] (0.7,0.1) -- (1.2,0.1);
	
	\draw[line width=1] (0,0.2) -- (0.3,0.2);
	\draw[line width=1] (0.4,0.2) -- (0.6,0.2);
	\draw[line width=1] (0.7,0.2) -- (1.2,0.2);
	
	\draw[line width=1] (0,0.3) -- (0.3,0.3);
	\draw[line width=1] (0.4,0.3) -- (0.6,0.3);
	\draw[line width=1] (0.7,0.3) -- (1.2,0.3);
	
	\draw[line width=1] (0.4,0.4) -- (0.6,0.4);
	\draw[line width=1] (0.7,0.4) -- (1.2,0.4);
	
	\draw[line width=1] (0.4,0.5) -- (0.6,0.5);
	
	\draw[line width=1] (0.4,0.6) -- (0.6,0.6);
	
	\draw[->] (1.22,0.01) arc (-80:80:0.04);
	\draw[->] (1.22,0.11) arc (-80:80:0.04);
	\draw[->] (1.22,0.21) arc (-80:80:0.04);
	\draw[->] (1.22,0.31) arc (-80:80:0.04);
	
	\node at (1.34,0.05) {$T^{(r)}$};
	\node at (1.34,0.15) {$T^{(r)}$};
	\node at (1.34,0.25) {$T^{(r)}$};
	\node at (1.34,0.35) {$T^{(r)}$};
	
	\draw[->] (1.22,0.41) arc (80:-150:0.6 and 0.3);
	\node at (1.64,0.1) {$T^{(r)}$};
		
	\node[red] at (0.15,-0.05) {\small $I^{(k)}_{1}$};
	\node[red] at (0.5,-0.06) {\small $I^{(k)}_{2}$};
	\node[red] at (0.95,-0.05) {\small $I^{(k)}_{3}$};
	
	\node at (0.95,0.46) {\small $\Big(T^{(r)}\Big)^{q^{(r,k)}_3-1} \Big(I^{(k)}_{3}\Big)$};
	
	\draw[<->,dashed] (-0.07,0) -- (-0.07,0.3);
	\node[text width = 1.8 cm] at (-0.2,0.15) {$\#$ of floors \\ $= q_1^{(r,k)}$}; 
	
	\draw[blue] (0.38,-0.02) -- (0.62,-0.02) -- (0.62,0.62) -- (0.38,0.62) -- cycle;
	\node[blue] at (0.3,0.53) {$Z^{(r,k)}_2$};
	
\end{tikzpicture}
\caption{An example of towers for an IET on 3 intervals. The union of the base intervals is $I^{(k)}$, and the union of all the floors  (including the base) is $I^{(r)}$. From the towers we can not see exactly the image under $T^{(r)}$ of the top floor of a tower, but we do know that it gets sent into $I^{(k)}$, hence somewhere in the union of the bases. }
\end{figure}

We say an IET $T$ is of \emph{periodic type} if it is a fixed point of $\mathcal{R}^N$ for some $N$. For Keane IETs it follows automatically that some power of the matrix $A^{(0,N)}$ is positive, hence we do not need to include this condition as part of the definition.\\

Given any function $g$ on $I$, denote its Birkhoff sums with respect to $T$ by $S^T_n g(x) := \sum_{k=0}^{n-1} g(T^k x)$.

Consider a piecewise constant function $\phi$ on $I$, such that $\phi$ is constant on every $I_j$ for $j \in \mathcal{A}$. The following is a standard fact, allowing one to compute the Birkhoff sums of $\phi$ at special times using the Rauzy-Veech cocycle. For a proof, see for example \cite[Lemma 3.10]{Tum}.

\begin{lemma}\label{lem:cocycle}
Let $\phi$ be a piecewise constant function, constant on each $I_j$, let $\phi^{(k)} = (A^{(0,k)})^T \phi$. Then for any $j\in \mathcal{A}$ and any $k>0$, for $x\in I^{(k)}_j$, $S^T_{q_j^{(0,k)}} \phi (x) = \phi^{(k)} (x)$.
\end{lemma}

Rauzy-Veech renormalisation is a discretisation of the Teichmüller geodesic flow, and the Rauzy-Veech cocycle is a discretisation of the Kontsevich-Zorich cocycle. (See \cite{Zorich}.) In particular, a translation surface $(M,\omega)$ being $g_t$-periodic implies that any IET which is a Poincaré section of the vertical flow is periodic under Rauzy-Veech renormalisation.

\subsection{Stable space and coboundaries}
Let $I$ be a horizontal interval on $M\setminus \Sigma$ that has no self-intersections. The first return map of the flow to $I$ is an IET $T:I \to I$, label its continuity intervals by $\{I_\alpha : \alpha \in \mathcal{A}\}$. For $\alpha \in \mathcal{A}$ denote by $\xi_\alpha \in H_1(M,\Z)$ the homology class of the loop obtained by following the flow $(\varphi^\theta_t)_{t\in\R}$ from any point $x\in I_\alpha$ until $T(x)$ and then returning to $x$ along the segment $[x,T(x)] \subset I$.
Given a homology class $\gamma \in H_1(M,\R)$, we can define a corresponding piecewise constant function $\phi_\gamma$ on $I$ by
\[\phi_\gamma|_{I_\alpha} := \langle \gamma, \xi_\alpha \rangle.\]

Let $\Phi: H_1(M,\R) \to \R^n$ be the functions that assigns to the homology class $\gamma$ the piecewise constant function $\phi_\gamma$ as defined above.

Note that if $\gamma \in H_1(M,\Z)$, then $\Phi(\gamma) \in \Z^n$.\\

The following combination of existing results, connecting negative Lyapunov exponents and coboundaries for the Poincaré section, is the key result that we apply.

In \cite{MMY1}, Marmi, Moussa and Yoccoz introduced a full measure Diophantine condition on IETs called \emph{Roth type}. We will not reproduce the definition here since we will not need to use it directly and refer the interested reader to \cite{MMY1}.

\begin{thm}[See Theorem A in \cite{MMY1} and Corollary 3.6 in \cite{MMY2}]\label{thm:MMY}
Let $T$ be an IET of Roth type. Suppose $\psi$ is a piecewise function that is constant on the intervals of continuity of $T$. Suppose that $\psi$ is a coboundary, then there exists a continuous function $h$ such that $\psi = h \circ T -h$.
\end{thm}

\begin{thm}\label{thm:coboundary}
Suppose that $(M,\omega)$ is either BOM generic or $g_t$-periodic. Then there exists a horizontal interval $I \subset M\setminus \Sigma$ with no self-intersections, such that the first return map $T: I \to I$ of the vertical flow to $I$ is a minimal ergodic IET, and for every $\alpha \in E^-_\omega$, the piecewise-constant function $\Phi(\alpha)$ on $I$ is a coboundary, with a continuous transfer function $h: I \to \R$ such that $\Phi(\alpha) (x) = h(T(x)) - h(x)$.
\end{thm}
\begin{proof}
In the case that $(M,\omega)$ is BOM generic, the fact that $\Phi(\alpha)$ is a coboundary is proven in \cite[Thm 4.2]{FU} (see also \cite[Thm 4.6]{FH}). It is left to check that the transfer function $h$ is in fact continuous, not just measurable. As shown in \cite[Section 1.2.2]{ChaikaEskin}, the IET $T$ is of Roth type, hence we can apply \Cref{thm:MMY} to conclude that $h$ is  continuous.\\

Consider now the case that $(M,\omega)$ is $g_t$-periodic, with $g_{T_0} \omega = \omega$. Then $(M,\omega)$ is automatically Oseledets and Masur generic. Indeed in this case Oseledets genericity is immediate, as the Kontsevich-Zorich cocycle, at times which are multiples of $T_0$, is given by a constant matrix. Masur genericity in the periodic case is also classical, following from Masur's criterion (\cite{Masur92}).

Thus the only assumption that is not satisfied is Birkhoff genericity. But in the proof as written in \cite[Thm 4.5]{FH}, Birkhoff genericity is used only to show that for some compact section $\mathcal{K}$, the consecutive hitting times $(t_k)_{k\geq 0}$ by the Teichmüller flow $(g_t \omega)_{t\in \R}$ of $\mathcal{K}$ satisfy $\frac{t_k}{k} \to 0$. But in the periodic case, if the closed geodesic enters $\mathcal{K}$ $n$ times, then $t_{n+k} = t_k + T_0$, and so the limit $\lim_{k \to \infty} \frac{t_k}{k} = 0$ is satisfied. 
The rest of the proof carries through without alterations. 

Finally, all periodic type IETs are Roth type (see \cite[Section 1.3.4]{MMY1}), so here we can also apply \Cref{thm:MMY} to conclude.
\end{proof}

\section{Covers}\label{sec:covers}
For a manifold $M$, a \emph{$\Z^d$-cover of $M$} is a covering map $p:\tilde M \to M$ equipped with a Deck group action by $\Z^d$. We will denote the cover simply by $\tilde M$, with the covering map and the action of $\Z^d$ implicit.

Denote by $\langle \cdot, \cdot \rangle$ the intersection form on $H_1(M,\R)$.

\begin{lemma}[See \cite{HooperWeiss} or Section 2 in \cite{FU}]\label{lem:covers}
The $\Z^d$-covers of $M$ are in one-to-one correspondence with $H_1(M,\Z)^d$, in such a way that if we denote the cover corresponding to $\gamma \in H_1(M,\Z)^d$ by $\tilde M_\gamma$, then $\tilde M_\gamma$ has the following property: 

If $\sigma$ is a closed curve on $M$ and 
\[n = (n_1,\dots,n_d) = (\langle \gamma_1, [\sigma] \rangle, \dots , \langle \gamma_d, [\sigma] \rangle),\]
then any lift of $\sigma$, $\tilde \sigma : [t_0,t_1] \to \tilde M_\gamma$ satisfies
\[\tilde \sigma (t_1) = n \cdot \tilde \sigma (t_0),\]
where $\cdot$ denotes the Deck group action of $\Z^d$.
\end{lemma}

Given a compact translation surface $(M,\omega)$ and a $\Z^d$-cover $\tilde M_\gamma$ of $M$, the pullback $\tilde \omega_\gamma = p^*\omega$ determines a translation structure on $\tilde M$. We say that $(\tilde M_\gamma, \tilde \omega_\gamma)$ is a \emph{$\Z^d$-cover} of the translation surface $(M,\omega)$.

We consider the vertical flow $(\tilde\varphi_t)^v_{t\in \R}$ on $(\tilde M_\gamma, \tilde \omega_\gamma)$. 

As a consequence of \Cref{lem:covers}, the following lemma holds.
Let $\tilde I = p^{-1}(I) \cong I \times \Z^d$. Denote by $\tilde T: I \times \Z^d \to I \times \Z^d$ the Poincaré map of $(\tilde \varphi_t)^v_{t\in \R}$ to $\tilde I$. 
\begin{lemma}[See Lemma 2.1 in \cite{FU}]\label{lem:skew-product}
The map $\tilde T$ is given by a skew-product $\tilde T = T_{\phi_\gamma}: I \times \Z^d \to I \times \Z^d$, where
\[T_{\phi_\gamma}(x,a) = (T(x),a+\phi_\gamma(x)).\]
\end{lemma}

\subsection{The wind-tree model}\label{sec:wind-tree}
In this section we define the wind-tree model and discuss the splitting of its homology, following \cite{DHL}.

Fix parameters $a,b \in (0,1)$. The wind-tree model with parameters $a,b$ is defined as
\[W(a,b) = \Z^2 \setminus \bigcup_{(k,l) \in \Z^2} \left(\left[k-\frac{a}2, k+\frac{a}2\right] \times \Big[l-\frac{b}2, l+\frac{b}2\Big]\right).\]
Unfolding the billiard in $W(a,b)$, one gets a translation surface $X_\infty = X_\infty(a,b)$, which is a $\Z^2$-cover as described below.\\

Let $Y = Y(a,b)$ be a genus two translation surface defined by gluing the opposite edges of the square with a rectangular hole
$\left(\big[-\frac{1}2,\frac{1}2\big]\times \big[-\frac{1}2,\frac{1}2\big]\right)\, \big\backslash \,  \left(\big[-\frac{a}2,\frac{a}2\big]\times \big[-\frac{b}2,\frac{b}2\big]\right)$.

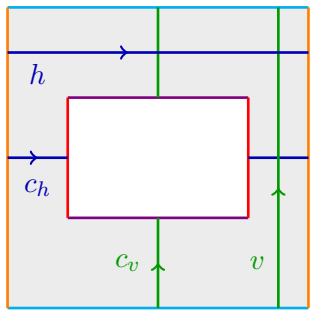
\begin{figure}[H]
\centering
\begin{tikzpicture}[scale=4,line width=1]
\fill[gray!15] (0,0) -- (1,0) -- (1,1) -- (0,1) -- (0,0) -- (0.2,0.3) -- (0.2,0.7) -- (0.8,0.7) -- (0.8,0.3) -- (0.2,0.3) -- cycle;
\draw[cyan] (0,0) -- (1,0);
\draw[cyan] (0,1) -- (1,1);
\draw[violet] (0.2,0.3) -- (0.8,0.3);
\draw[violet] (0.2,0.7) -- (0.8,0.7);
\draw[orange] (0,0) -- (0,1);
\draw[orange] (1,0) -- (1,1);
\draw[red] (0.2,0.3) -- (0.2,0.7);
\draw[red] (0.8,0.3) -- (0.8,0.7);
\begin{scope}[decoration={markings, mark=at position 0.5 with {\arrow{>}}}]
\draw[green!60!black, postaction={decorate}] (0.5,0) -- (0.5,0.3);
\draw[green!60!black] (0.5,0.7) -- (0.5,1);
\node[green!60!black] at (0.4,0.15) {$c_v$}; 

\draw[blue!70!black, postaction={decorate}] (0,0.5) -- (0.2,0.5);
\draw[blue!70!black,] (0.8,0.5) -- (1,0.5);
\node[blue!70!black] at (0.1,0.4) {$c_h$}; 
\end{scope}

\begin{scope}[decoration={markings, mark=at position 0.4 with {\arrow{>}}}]
\draw[green!60!black, postaction={decorate}] (0.9,0) -- (0.9,1);
\node[green!60!black] at (0.83,0.15) {$v$}; 

\draw[blue!70!black, postaction={decorate}] (0,0.85) -- (1,0.85);
\node[blue!70!black] at (0.1,0.78) {$h$}; 
\end{scope}

\end{tikzpicture}
\caption{The surface $Y$ with the gluings and curves $v,h,c_h$ and $c_v$.}
\label{fig:Y}
\end{figure}

Let $h,v,c_h$ and $c_v$ be the four curves on $Y$ as in \Cref{fig:Y}. Then let $X = X(a,b)$ be the $\Z^2 \times \Z^2$ cover of $Y$ corresponding to the homology classes $([c_h]-[h], [c_v]-[v]) \in H_1(Y,\Z/2\Z)^2$.
In \Cref{fig:X} we show $X$ with its four singularities.

\begin{figure}[H]
\centering
\begin{tikzpicture}[scale=3.5,line width=1]
\begin{scope}[decoration={markings, mark=at position 0.6 with {\arrow{>}}}]

\fill[gray!15] (0,0) -- (1,0) -- (1,1) -- (0,1) -- (0,0) -- (0.2,0.3) -- (0.2,0.7) -- (0.8,0.7) -- (0.8,0.3) -- (0.2,0.3) -- cycle;
\draw[cyan, postaction={decorate}] (0,0) -- (1,0);
\draw[cyan, postaction={decorate}] (0,1) -- (1,1);
\draw (0.2,0.3) -- (0.8,0.3);
\draw (0.2,0.7) -- (0.8,0.7);
\draw[orange, postaction={decorate}] (0,0) -- (0,1);
\draw[orange, postaction={decorate}] (1,0) -- (1,1);
\draw (0.2,0.3) -- (0.2,0.7);
\draw (0.8,0.3) -- (0.8,0.7);

\node[circle,fill,inner sep=0.9mm,red] at (0.2,0.3) {};
\node[rectangle,fill,inner sep=1mm,violet] at (0.8,0.3) {};
\node[diamond,fill,inner sep=0.8mm,green!60!black] at (0.2,0.7) {};
\node[star,fill,inner sep=0.8mm] at (0.8,0.7) {};

\node[cyan] at (0.6,-0.1) {$h_{00}$};
\node[cyan] at (0.6,0.9) {$h_{00}$};
\node[orange] at (-0.1,0.15) {$v_{00}$};
\node[orange] at (0.9,0.15) {$v_{00}$};

\draw[green!60!black, postaction={decorate}] (0.4,0) -- (0.4,0.3);
\draw[green!60!black] (0.4,0.7) -- (0.4,1);
\node[green!60!black] at (0.25,0.15) {$c_{v,0}$}; 

\draw[blue!70!black, postaction={decorate}] (0,0.5) -- (0.2,0.5);
\draw[blue!70!black,] (0.8,0.5) -- (1,0.5);
\node[blue!70!black] at (0.1,0.4) {$c_{h,0}$}; 

\begin{scope}[shift={(1.3,0)}]
\fill[gray!15] (0,0) -- (1,0) -- (1,1) -- (0,1) -- (0,0) -- (0.2,0.3) -- (0.2,0.7) -- (0.8,0.7) -- (0.8,0.3) -- (0.2,0.3) -- cycle;
\draw[cyan, postaction={decorate}] (0,0) -- (1,0);
\draw[cyan, postaction={decorate}] (0,1) -- (1,1);
\draw (0.2,0.3) -- (0.8,0.3);
\draw (0.2,0.7) -- (0.8,0.7);
\draw[orange, postaction={decorate}] (0,0) -- (0,1);
\draw[orange, postaction={decorate}] (1,0) -- (1,1);
\draw (0.2,0.3) -- (0.2,0.7);
\draw (0.8,0.3) -- (0.8,0.7);

\node[circle,fill,inner sep=0.9mm,red] at (0.8,0.3) {};
\node[rectangle,fill,inner sep=1mm,violet] at (0.2,0.3) {};
\node[diamond,fill,inner sep=0.8mm,green!60!black] at (0.8,0.7) {};
\node[star,fill,inner sep=0.8mm] at (0.2,0.7) {};

\node[cyan] at (0.6,-0.1) {$h_{10}$};
\node[cyan] at (0.6,0.9) {$h_{10}$};
\node[orange] at (0.1,0.15) {$v_{10}$};
\node[orange] at (1.1,0.15) {$v_{10}$};

\draw[green!60!black, postaction={decorate}] (0.4,0) -- (0.4,0.3);
\draw[green!60!black] (0.4,0.7) -- (0.4,1);
\node[green!60!black] at (0.55,0.15) {$c_{v,1}$}; 

\draw[blue!70!black,] (0,0.5) -- (0.2,0.5);
\draw[blue!70!black,] (0.8,0.5) -- (1,0.5);
\end{scope}

\begin{scope}[shift={(1.3,1.3)}]
\fill[gray!15] (0,0) -- (1,0) -- (1,1) -- (0,1) -- (0,0) -- (0.2,0.3) -- (0.2,0.7) -- (0.8,0.7) -- (0.8,0.3) -- (0.2,0.3) -- cycle;
\draw[cyan, postaction={decorate}] (0,0) -- (1,0);
\draw[cyan, postaction={decorate}] (0,1) -- (1,1);
\draw (0.2,0.3) -- (0.8,0.3);
\draw (0.2,0.7) -- (0.8,0.7);
\draw[orange, postaction={decorate}] (0,0) -- (0,1);
\draw[orange, postaction={decorate}] (1,0) -- (1,1);
\draw (0.2,0.3) -- (0.2,0.7);
\draw (0.8,0.3) -- (0.8,0.7);

\node[circle,fill,inner sep=0.9mm,red] at (0.8,0.7) {};
\node[rectangle,fill,inner sep=1mm,violet] at (0.2,0.7) {};
\node[diamond,fill,inner sep=0.8mm,green!60!black] at (0.8,0.3) {};
\node[star,fill,inner sep=0.8mm] at (0.2,0.3) {};

\node[cyan] at (0.6,0.1) {$h_{11}$};
\node[cyan] at (0.6,1.1) {$h_{11}$};
\node[orange] at (0.1,0.15) {$v_{11}$};
\node[orange] at (1.1,0.15) {$v_{11}$};

\draw[green!60!black] (0.4,0) -- (0.4,0.3);
\draw[green!60!black] (0.4,0.7) -- (0.4,1);

\draw[blue!70!black,] (0,0.5) -- (0.2,0.5);
\draw[blue!70!black,] (0.8,0.5) -- (1,0.5);
\end{scope}

\begin{scope}[shift={(0,1.3)}]
\fill[gray!15] (0,0) -- (1,0) -- (1,1) -- (0,1) -- (0,0) -- (0.2,0.3) -- (0.2,0.7) -- (0.8,0.7) -- (0.8,0.3) -- (0.2,0.3) -- cycle;
\draw[cyan, postaction={decorate}] (0,0) -- (1,0);
\draw[cyan, postaction={decorate}] (0,1) -- (1,1);
\draw (0.2,0.3) -- (0.8,0.3);
\draw (0.2,0.7) -- (0.8,0.7);
\draw[orange, postaction={decorate}] (0,0) -- (0,1);
\draw[orange, postaction={decorate}] (1,0) -- (1,1);
\draw (0.2,0.3) -- (0.2,0.7);
\draw (0.8,0.3) -- (0.8,0.7);

\node[circle,fill,inner sep=0.9mm,red] at (0.2,0.7) {};
\node[rectangle,fill,inner sep=1mm,violet] at (0.8,0.7) {};
\node[diamond,fill,inner sep=0.8mm,green!60!black] at (0.2,0.3) {};
\node[star,fill,inner sep=0.8mm] at (0.8,0.3) {};

\node[cyan] at (0.6,0.1) {$h_{01}$};
\node[cyan] at (0.6,1.1) {$h_{01}$};
\node[orange] at (-0.1,0.15) {$v_{01}$};
\node[orange] at (0.9,0.15) {$v_{01}$};

\draw[green!60!black] (0.4,0) -- (0.4,0.3);
\draw[green!60!black] (0.4,0.7) -- (0.4,1);

\draw[blue!70!black, postaction={decorate}] (0,0.5) -- (0.2,0.5);
\draw[blue!70!black,] (0.8,0.5) -- (1,0.5);
\node[blue!70!black] at (0.1,0.4) {$c_{h,1}$}; 
\end{scope}

\end{scope}

\end{tikzpicture}
\caption{The surface $X$ with its four singularities and the classes generating $H_1(X,\Z)$.}
\label{fig:X}
\end{figure}
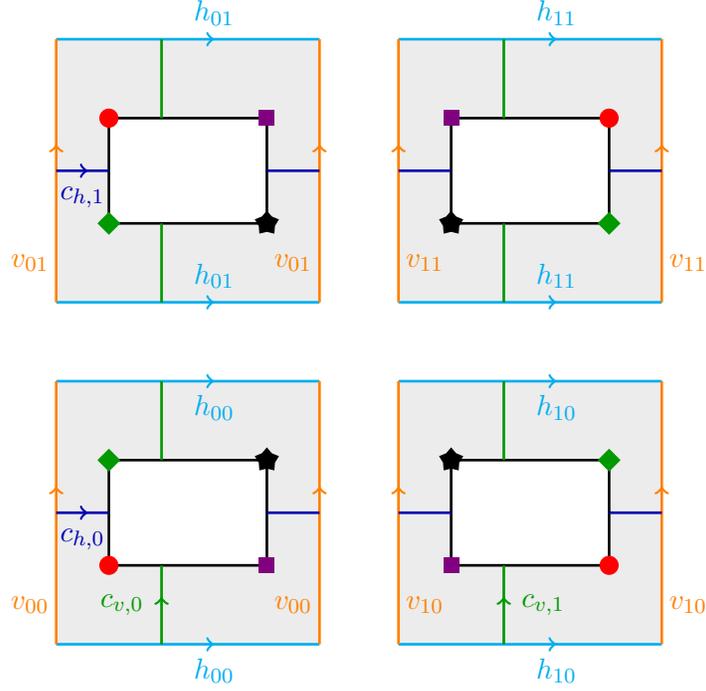

Define the homology classes $h_{xy}, v_{xy} \in H_1(X,\Z)$ for $x,y\in \{0,1\}$ as in \Cref{fig:X}.
If we define 
\begin{align*}
\gamma_h &:= -v_{00}+v_{10}-v_{01}+v_{11},\\
\gamma_v &:= h_{00}+h_{10}-h_{01}-h_{11},
\end{align*}
then $X_\infty$ is the $\Z^2$-cover arising from the classes $(\gamma_h,\gamma_v)$.\\

One computes that $X$ has genus 5, and so $H_1(X,\Z) \cong \Z^{10}$. The twelve homology classes $h_{xy},v_{xy}, c_{h,y}, c_{v,x}$ for $x,y \in \{0,1\}$ generate $H_1(X,\Z)$ with the following two relations:
\begin{align*}
c_{h,0} - c_{h,1} &= h_{00} + h_{10} - h_{01} - h_{11}\\
c_{v,0} - c_{v,1} &= v_{00} - v_{10} + v_{01} - v_{11}.
\end{align*}

Now consider the Deck group action of the Klein group $\Z/2\Z \times \Z/2\Z$ on $X$, seen as a cover of $Y$. If we let $\tau_h$ and $\tau_v$ be the generators of the two $\Z/2\Z$ factors, they act on $X$ by permuting the four sheets. In particular, the induced action on homology is given by:
\begin{align*}
(\tau_h)_* h_{xy} &= h_{(x+1)\, y}\\
(\tau_h)_* v_{xy} &= v_{(x+1)\, y}\\
(\tau_h)_* c_{h,y} &= c_{h,y}\\
(\tau_h)_* c_{v,x} &= c_{v,x+1},
\end{align*}
and symmetrically for $\tau_v$.

Since the action of $K$ commutes with the $SL(2,\Z)$ action on $\mathcal{M}_1(X)$, $H_1(X,\R)$ splits into four invariant blocks, 
\[H_1(X,\R) = E^{++} \oplus E^{+-} \oplus E^{-+} \oplus E^{--},\]
where
\[E^{st} = \{\xi \in H_1(X,\R): (\tau_h)_*\xi = s\cdot \xi, (\tau_v)_*\xi = t\cdot \xi\}\, \text{ for }\, s,t\in \{+,-\}.\]

These blocks have the following generators:
\begin{align*}
E^{++} &= \R (h_{00}+h_{10}+h_{01}+h_{00}) \oplus \R (v_{00}+v_{10}+v_{01}+v_{00}) \oplus \R (c_{h,0}+c_{h,1}) \oplus \R (c_{v,0}+c_{v,1}) \\ 
E^{+-} &= \R (\gamma_v) \oplus \R (v_{00}+v_{10}-v_{01}-v_{11})\\
E^{-+} &= \R (h_{00}-h_{10}+h_{01}-h_{11}) \oplus \R (\gamma_h)\\
E^{--} &= \R (h_{00}-h_{10}-h_{01}+h_{11}) \oplus \R (v_{00}-v_{10}-v_{01}+v_{11})
\end{align*}

Each of these four blocks is symplectic, and the Lyapunov exponents for the Kontsevich-Zorich cocycle over the corresponding invariant subbundle are respectively $\{1, \lambda^{++}, -\lambda^{++}, -1\}$, $\{\lambda^{+-}, -\lambda^{+-}\}$, $\{\lambda^{-+}, -\lambda^{-+}\}$, $\{\lambda^{--}, -\lambda^{--}\}$. 

\begin{thm}[See Corollary 1 in \cite{DHL}]\label{thm:LyapExp}
For any $a,b\in (0,1)$, for almost all $\theta \in S^1$, the following holds.
 
Let $\mathcal{N}$ be the $SL(2,\R)$ orbit closure of $r_{\pi/2-\theta}X(a,b)$ and $\nu_\mathcal{N}$ the corresponding invariant probability measure. 
Then the Lyapunov exponents for the Kontsevich-Zorich cocycle with respect to $\nu_\mathcal{N}$ are given by $\lambda^{++} = \lambda^{--} = 1/3$ and $\lambda^{+-} = \lambda^{-+} = 2/3$.
\end{thm}

\section{Construction of invariant sets and non-transitivity}\label{sec:invariant}

\begin{thm}\label{thm:invariant}
Suppose that $(M,\omega)$ is either BOM generic or $g_t$-periodic. Suppose that $H_1(M,\Z) = K \oplus K^\perp$ is an orthogonal splitting with respect to the symplectic form, and that $V = K \otimes \R$ defines an $SL_2(\R)$-invariant subbundle $\mathcal{V}$ of the homology Hodge 
bundle of $(M,\omega)$. 
Let $2d = \dim K$, and let $d_+$ be the number of positive Lyapunov exponents of $\mathcal{V}$. Assume that $d_+ > 0$. Let $m=2d-d_+$ be the number of non-positive Lyapunov exponents, and let $\gamma_1,\dots,\gamma_m \in K$ be any linearly independent classes.
Let $(\tilde M_\gamma,\tilde \omega_\gamma)$ be the $\Z^m$-cover corresponding to $\gamma = (\gamma_1,\dots,\gamma_m)$.

Then there exists a lattice $\Lambda \subset \R^{d_+}$ and a non-constant continuous function $\hat h:\tilde M_\gamma \to \R^{d_+}/\Lambda$, which is invariant for the vertical flow $\tilde \varphi^v_t$ on $(\tilde M_\gamma,\tilde \omega_\gamma)$.
\end{thm}

\begin{proof}
Let $I$ be a horizontal Poincaré section for the vertical flow on $(M,\omega)$ satisfying the conclusion of \Cref{thm:coboundary}, and let $T:I\to I$ be the Poincaré map. Let $\phi:I \to \Z^m$ be the piecewise constant function given by $\phi = (\Phi(\gamma_1),\dots,\Phi(\gamma_m)).$ 
By \Cref{lem:skew-product} the vertical flow on $(\tilde M_\gamma, \tilde \omega_\gamma)$ is a special flow over the skew-product $T_\phi$, hence to construct an invariant function for the vertical flow we construct an invariant function for $T_\phi$ on $I \times \Z^m$ and then extend it to $\tilde M_\gamma$ along vertical leaves.\\

As the numbers of positive and negative exponents are the same by the symplectic structure of the cocycle, the space of stable vectors in $V$, $V \cap E^-_\omega$ is $d_+$-dimensional. Let $\alpha_1,\dots,\alpha_{d_+}$ be a basis for this stable subspace of $V$.

Extend the set $\{\gamma_1,\dots,\gamma_m\}$ to a basis $\{\gamma_1,\dots,\gamma_m,\sigma_1,\dots,\sigma_{d_+}\}$ of $K$.  

Express the $\alpha_i$ in terms of this basis as
\[\alpha_i = \sum_{j=1}^m b_{ij} \gamma_j + \sum_{k=1}^{d_+} c_{ik} \sigma_k,\]
with $b_{ij}, c_{ik} \in \R$.

Since the $\alpha_i$ are stable, by \Cref{thm:coboundary} the functions $\psi_i = \Phi(\alpha_i)$ are coboundaries, i.e. there exist continuous 
$h_1,\dots,h_{d_+}: I \to \R$ such that 
\[h_i(Tx)-h_i(x) = \psi_i(x) = \sum_{j=1}^m b_{ij} \phi_j(x) + \sum_{k=1}^{d_+} c_{ik} \Phi(\sigma_k)(x).\]

Now define the function $\tilde h: I\times Z^m \to \R^{d_+}$, where for $a=(a_1,\dots,a_m) \in \Z^m$,
\[\big[\tilde h(x,a)\big]_i = h_i(x)-\sum_{j=1}^m b_{ij} a_j\ \ \ \text{for}\ i=1,\dots,d_+.\]

Then observe that
\begin{align*}
\big[\tilde h(T_\phi(x,a))\big]_i &= h_i(Tx) - \sum_{j=1}^m b_{ij} (a_j+\phi_j(x))\\
&= h_i(x) + \sum_{j=1}^m b_{ij} \phi_j(x) + \sum_{k=1}^{d_+} c_{ik} \Phi(\sigma_k)(x) - \sum_{j=1}^m b_{ij} a_j - \sum_{j=1}^m b_{ij} \phi_j(x)\\
&= \big[\tilde h(x,a)\big]_i + \sum_{k=1}^{d_+} c_{ik} \Phi(\sigma_k)(x).
\end{align*}

Now note that since $\sigma\in H_1(M,\Z)$, $\Phi(\sigma_k)(x) \in \Z$ for all $k,x$.

Hence if $C \in M_{d_+\times d_+}$ is the matrix with entries $c_{ik}$, we see that for all $(x,a)\in I\times \Z^m$, 
\[\tilde h(T_\phi(x,a)) - \tilde h(x,a) = \bigg(\sum_{k=1}^{d_+} c_{1k} \Phi(\sigma_k)(x), \dots , \sum_{k=1}^{d_+} c_{d_+ k} \Phi(\sigma_k)(x)\bigg) = 
C \cdot \begin{pmatrix} \Phi(\sigma_1)(x) \\ \vdots \\ \Phi(\sigma_{d_+})(x) \end{pmatrix}  \in C \cdot \Z^{d_+}.\]

Let $\Lambda = C \cdot \Z^{d_+}$, which is a lattice in $\R^{d_+}$. Then if we define $\hat h : I \times Z^m \to \R^{d_+}/\Lambda$ by composing $\tilde h$ with 
the quotient map, $\hat h$ is a continuous function and it is $T_\phi$-invariant.\\

Finally, let us show that $\hat h$ can not be constant on $I\times \Z^m$. Indeed, suppose $\hat h$ is constant. Then since $\tilde h$ is continuous, also $\tilde h$ must be constant. By the definition of $\tilde h$, considering a fixed $a\in \Z^m$, we see that every $h_i$ must be constant on $I$. But this implies that $\alpha_i = 0$, which contradicts our assumption. Thus we conclude that $\hat h$ is a non-constant, continuous, $T_\phi$-invariant function on $I\times \Z^m$. 

If we extend $\hat h$ to $\tilde M_\gamma$, by defining, for every $\tilde x \in I$, the value on the vertical segment joining $\tilde x$ and $T_\phi (\tilde x)$ to be equal to the value at $\tilde x$, we obtain a non-constant continuous function on $\tilde M_\gamma$ that is invariant under the vertical flow.
\end{proof}

\begin{cor}
Under the assumptions of \Cref{thm:invariant}, the vertical flow on $(\tilde M_\gamma,\tilde\omega_\gamma)$ is non-transitive.
\end{cor}
\begin{proof}
Observe that by \Cref{thm:invariant}, for any $(x,a) \in I \times \Z^m$,
\[\overline{\{T_\phi^n(x,a): n \in \Z\}} \subset \overline{\hat h^{-1}(\hat h(x,a))}.\]
As $\hat h$ is continuous, in fact the level set $\hat h^{-1}(\hat h(x,a))$ is closed.
Since $\hat h$ is not constant, for every $(x,a)$ this level set is not all of $I \times \Z^m$, and so no orbit of $T_\phi$ is dense. Hence also no orbit of $\tilde \varphi^v_t$ is dense on $(\tilde M_\gamma,\tilde \omega_\gamma)$.
\end{proof}

\section{Hausdorff dimension in the periodic case}\label{sec:hdim}
Here we consider the case where the IET $T$ is of periodic type.
Let $N$ be such that $\mathcal{R}^N (T) = T$, and the corresponding matrix $A = A^{(0,N)}$ is positive. Denote by $T_k$ the induced IET $\hat {\mathcal{R}}^{Nk}(T)$, defined on an interval that we call $J^{(k)}$. Denote by $x_0$ the left endpoint of $I$, which is also the left end-point of $J^{(k)}$.

Denote by $Z_j := Z_j^{(0,N)}$ the Rokhlin tower for $T$ over $J_j^{(1)}$, let $q_j$ be its height.
By the periodicity, $Z_j$ is naturally identified with the Rokhlin tower for $T_{k-1}$ over $J^{(k)}_j$ for any $k$.\\

Suppose that $\psi$ is an eigenvector of $A^T$ with eigenvalue $\lambda < 1$.
We denote by $\psi$ also the piecewise constant function $\psi(x) = \psi_\alpha$ for $x\in I_\alpha$. Let $h$ be the continuous transfer function satisfying $\psi(x) = h(T(x)) - h(x)$, we assume without loss of generality that $h(x_0) = 0$.

\begin{lemma}\label{lem:induced_h}
For $x\in J_j^{(1)}$, $h(T_1(x)) - h(x) = \lambda \psi_j$.
\end{lemma}
\begin{proof}
Observe that $T_1(x) = T^{q_j}(x)$, hence
\[h(T_1(x)) - h(x) = \sum_{i=0}^{q_j-1} h(T^{i+1}(x)) - h(T^i(x)) = S^T_{q_j} \psi (x) = \psi_j^{(N)} = A^T \psi_j = \lambda \psi_j,\]
where in the third equality we used \Cref{lem:cocycle}.
\end{proof}

\subsection{The Vershik adic coding for IETs}
Here we discuss the Vershik coding for IETs, introduced in \cite{Bufetov} and \cite{LT}.
We give only the brief definitions necessary for the following, for a more detailed explanation and proofs see the above references or \cite{Tum}.

The idea is to code a point $x \in I$ by the floor of the tower $Z_{j_k}^{(0,Nk)}$ in which it appears. Taking larger $k$ gives a finer partition of $I$, and recording the floor for all $k\in \N$ allows one to identify $x$ uniquely.\\

Let $\mathcal{E}$ be the set of all floors of the towers $\{Z_j: j \in \mathcal{A}\}$, where we denote the floor $T^\ell(J^{(1)}_j)$ of $Z_j$ by $(j,\ell)$:
\[\mathcal{E} = \{(j,\ell): j\in \mathcal{A}, 0 \leq \ell < q_j\}.\]

The set $\mathcal{E}$ is the set of edges in a graph called the \emph{Bratteli diagram}, hence by convention we will sometimes call the elements of $\mathcal{E}$ \emph{edges}, and talk about their \emph{start} $s(e)$ and \emph{terminus} $t(e)$.

For an edge  $e=(j,l)$ we define $t(e) = j$ and $s(e)=i$, where $i\in\mathcal{A}$ is such that $I_i \supset T^\ell(J^{(1)}_j)$. Importantly, by periodicity, also for any $k \geq 1$, $J^{(k-1)}_i \supset T_{k-1}^\ell(J^{(k)}_j)$.

Denote by $\Sigma_k \subset \mathcal{E}^k$ the \emph{admissible paths} of length $k$, i.e. sequences $(e_1,\dots,e_k)$ such that $t(e_1) = s(e_2), \dots, t(e_{k-1}) = s(e_k)$.

Given a path $p_k = ((j_1,\ell_1),\dots,(j_k,\ell_k)) \in \Sigma_k$, we define a corresponding subinterval of $I$, by
\[\mathcal{J}(p_k) = T^{l_1} \circ T_1^{l_2} \circ \dots \circ T_{k-1}^{l_k} \Big(J^{(k)}_{j_k}\Big).\]

\newpage
\begin{lemma}[See Lemma 4.7 and Lemma 4.8 in \cite{Tum}]\hspace{1cm} 
\begin{enumerate} 
\item
The map $\mathcal{J}$ is a bijection between the set $\Sigma_k$ of admissible paths of length $k$ and floors of the towers for $T$ over $J^{(k)}$.
\item If $p_{k+1} \in \Sigma_{k+1}$ and $p_k$ is the path consisting of the first $k$ entries of $p_{k+1}$, then $\mathcal{J}(p_{k+1}) \subset \mathcal{J}(p_k)$.
\end{enumerate}

Hence if $\Sigma$ denotes the space of \emph{infinite admissible paths},
\[\Sigma := \{(e_1,e_2,\dots): \forall k \geq 1, t(e_k) = s(e_{k+1})\},\]
and denoting by $p_k\in \Sigma_k$ the first $k$ entries of $p\in \Sigma$, one can define the map $\mathcal{J}: \Sigma \to I$ by
\[\mathcal{J}(p) := \cap_{k \geq 1} \mathcal{J}(p_k).\]

The map $\mathcal{J}$ is surjective, it is 2-to-1 on countably many sequences in $\Sigma$ and one-to-one everywhere else.
\end{lemma}

\subsection{Computing the transition function $h$ using the coding}\label{sec:h_formula}
Define the function $f: \mathcal{E} \to \R$ to be the partial Birkhoff sum of $\psi$ from the base of the tower $Z_j$ until the floor corresponding to the edge $(j,\ell)$:
\[f((j,\ell)) := S^T_\ell \psi |_{I_j}.\]

\begin{prop}\label{prop:h_formula}
If a point $x$ has coding $x = \mathcal{J}(p)$ for $p=(e_1,e_2,\dots) \in \Sigma$, the value of $h$ at $x$ is given by the series
\[h(x) = \sum_{i=1}^{\infty} \lambda^{i-1} f(e_{i}).\]
\end{prop}
\begin{proof}
Consider the finite truncations $p_k = (e_1,\dots,e_k)$ for $k \geq 1$. Let $x_k$ be the left endpoint of the interval in $I$ corresponding to $p_k$. Since $\mathcal{J}(p_k)$ is a floor of the tower above $J^{(k)}_{j_k}$, the point $x_k$ is in the orbit of the left endpoint $y^{(k)}_{j_k}$ of $J^{(k)}_{j_k}$,
\[x_k = T^{\ell_1} \circ T_1^{\ell_2} \circ \dots \circ T_{k-1}^{\ell_k}(y^{(k)}_{j_k}).\]

Denote by $x_{k,r}$ the point $x_{k,r} = T_{r}^{\ell_{r+1}} \circ T_{r+1}^{\ell_{r+2}} \circ \dots \circ T_{k-1}^{\ell_k}(y^{(k)}_{j_k})$, so that $x_{k,0} = x_k$, and 
$x_{k,k} = y^{(k)}_{j_k}$. As $x_{k,i} = T_{i}^{\ell_{i+1}}(x_{k,i+1})$, applying \Cref{lem:induced_h}, we see that
\[h(x_{k,i}) - h(x_{k,i+1}) = S^{T}_{\ell_{i+1}} \lambda^i\psi |_{I_{j_{i+1}}} = \lambda^i f(e_{i+1}).\]

Hence
\[h(x_k) = h(x_{k,k})+\sum_{i=0}^{k-1} \big(h(x_{k,i}) - h(x_{k,i+1})\big) = h(x_{k,k})+ \sum_{i=1}^{k} \lambda^{i-1} f(e_i).\]

Now observe that $\lim_{k \to \infty} x_{k,k} = x_0$, and so $\lim_{k \to \infty} h(x_{k,k}) = 0$. Similarly, as $\lim_{k\to \infty} x_k = x$, 
\[h(x) = \lim_{k\to \infty} h(x_k) = \sum_{i=1}^{\infty} \lambda^{i-1} f(e_{i}).\]
\end{proof}

\subsection{Bounding the Hausdorff dimension}
\begin{lemma}\label{lem:different_sums}
By taking a multiple of the period $N$ if needed, we can assume that for any $i,j$, there exist two edges $e_1,e_2\in \mathcal{E}$ with $s(e_r) = i, t(e_r) = j$ for $r=1,2$, and $f(e_1) \neq f(e_2)$.
\end{lemma}
\begin{proof}
Fix a letter $i$. Note that $h$ can not be constant on $I_i$, since then it would be piecewise constant everywhere, and since it is continuous it would have to be constant on $I$, which is not possible as $\psi \neq 0$. Then there must exist two subintervals $U_{i,1}$ and $U_{i,2}$ of $I_i$, such that $h(U_{i,1}) \cap h(U_{i,2}) = \emptyset$. 

By minimality of $T$, there are points $y_{i,1}\in U_{i,1}$ and $y_{i,2}\in U_{i,2}$, such that $y_{i,2} = T^k(y_{i,1})$ for some $k$. Since by assumption $h(U_{i,1}) \cap h(U_{i,2}) = \emptyset$, it must hold that $S^T_k \psi (y_{i,1}) \neq 0$, since it is equal to the difference $h(y_{i,2})-h(y_{i,1})$. 

Let $J_{i,1}$ be the interval containing $y_{i,1}$ on which $T^k$ is continuous, and let $J_{i,2} = T^k(J_{i,1})$. Take $N_i$ to be a sufficiently high multiple of the period $N$ such that, for each $j$, in the tower $Z^{(0,N_i)}_j$ there is a floor which is a subinterval of $J_{i,1}$ and $k$ floors above it a floor that is a subinterval of $J_{i,2}$. Then these two floors correspond to two edges $e_1, e_2$ with 
start $i$ and end $j$, and $f(e_2)-f(e_1) = S_k \psi (y_{i,1}) \neq 0$.
Doing this for each $i$ and taking $N = \max_i N_i$ will give an $N$ that works for all $i$ and $j$.
\end{proof}

For fixed $i,j$, call the two edges $e_1,e_2$ as above an \emph{alternate pair} for $i,j$.
From now on we assume that $N$ is large enough to satisfy the conclusion of \Cref{lem:different_sums}. For each $i,j$, let $\delta_{ij}$ be the difference $|f(e_2)-f(e_1)| > 0$ for the alternate pair $e_1,e_2$. Let $\delta = \min_{i,j} \delta_{ij}$.

Let $F = \max_{(e,e')\in \mathcal{E}^2} |f(e)-f(e')|$ and let $b$ be an integer such that 
\[\frac{F \lambda^b}{1-\lambda}  <\delta.\]

\begin{lemma}\label{lem:gaps}
Let $(e_1,e'_1)$ be an alternate pair, and let $(e_1,\dots,e_b)$ be an admissible finite path. Then $h(\mathcal{J}(e_1,e_2,\dots,e_b)) \cap h(\mathcal{J}(e'_1,e_2,\dots,e_b)) = \emptyset$.
\end{lemma}
\begin{proof}
Observe that if $p=(\tilde e_1,\tilde e_2,\dots),q=(\hat e_1,\hat e_2,\dots)$ are two paths such that $\tilde e_i = \hat e_i$ for $i=1,\dots,b$, then 
\[|h(p)-h(q)| \leq F \sum_{i={b+1}}^\infty \lambda^{i-1} \leq \frac{F\lambda^b}{1-\lambda} < \delta.\] 
Now take any path $p\in \Sigma$ starting with the finite path $(e_1,e_2,\dots, e_b)$ and any path $p'\in \Sigma$ starting with $(e'_1,e_2,\dots,e_b)$.
Then if $q$ is the path obtained from $p'$ by replacing the first edge $e'_1$ with $e_1$, then $|h(q) - h(p')|  = |f(e_1)-f(e'_1)| > \delta$, whereas $|h(p) - h(q)| < \delta$. Thus $h(p) \neq h(p')$, and hence
$h(\mathcal{J}(e_1,\dots,e_b)) \cap h(\mathcal{J}(e'_1,e_2,\dots,e_b)) = \emptyset$.
\end{proof}

In particular, for any given value $z$, at least one of the two intervals $\mathcal{J}(e_1,\dots,e_b)$ and $\mathcal{J}(f_1,e_2,\dots,e_b)$ must be in the complement of $h^{-1}(z)$. 
This allows us to find gaps in the level set $h^{-1}(z)$.

\begin{thm}\label{thm:hdim}
Let $T, h$ as above. Then for any $z\in \R$, the Hausdorff dimension of the level set $h^{-1}(z)$ is strictly less than 1.
\end{thm} 
\begin{proof}
Fix some $z\in \R$. We will show that $h^{-1}(z)$ is contained in a Cantor set of Hausdorff dimension less than 1. As shown in 
\Cref{lem:gaps}, picking any alternate pair $(e_1,f_1)$ and admissible finite path $(e_1,\dots,e_b)$ gives us a gap in $h^{-1}(z)$, as either 
$\mathcal{J}(e_1,e_2,\dots,e_b)$ or $\mathcal{J}(f_1,e_2,\dots,e_b)$ must be in the complement of the level set. 

For any $k \geq 1$, recall that we can partition $I$ as $I = \sqcup_{p\in \Sigma_{kb}} \mathcal{J}(p)$, where $\Sigma_{kb}$ are the finite admissible paths of length $kb$. For any $p=(e_1,e_2,\dots,e_{kb}) \in \Sigma_{kb}$, there exists an alternate pair $(e_{kb+1},f_{kb+1})$ and an admissible path $(e_{kb+1},e_{kb+2},\dots,e_{(k+1)b})$. To this pair corresponds as above a gap in $h^{-1}(z)$, which we will denote by $G_p$. Thus there is a gap $G_p \subset I \setminus h^{-1}(z)$ for every element $p$ of the partition. The choice of alternate pair and admissible path that gives a gap is very far from unique, but we just need to find one. We will construct our Cantor set $C$ inductively as $C = \cap_{k \geq 1} C_k$, where $C_{k+1}$ is obtained from $C_{k}$ by removing the gap $G_p$ from every element of the partition $I = \sqcup_{p\in \Sigma_{kb}} \mathcal{J}(p)$, which is still contained in $C_k$.

More formally, for $k=0$, denote by $G_i$ a gap inside $I_i$, given by an alternate pair $(e_1,e'_1)$ with $s(e_1)=s(e'_1) = i$, and some admissible path $(e_1,\dots,e_b)$.
Let $C_1 = I\setminus \cup_{i\in \mathcal{A}} G_i$.
For $k \geq 1$, let $\mathcal{P}_k$ be the set of admissible paths $p = (e_1,\dots,e_{kb})$ such that $\mathcal{J}(p) \subset C_k$, and define
$C_{k+1} = C_k \setminus \cup_{p \in \mathcal{P}_k} G_p$. Let $C = \cap_{k \geq 1} C_k$. Then $h^{-1}(z) \subset C$.\\

Now let us show that the Hausdorff dimension of $C$ is strictly less than 1. First let us bound the number of gaps in $C_k$. Let $m = \max_{ij} A^b_{ij}$. 
Then the number of all admissible paths of length $kb$ is bounded above by $n m^k$, where $n=|\mathcal{A}|$. Hence $\#\mathcal{P}_k \leq n m^k$, and so the number of gaps in $C_k$ is bounded above by $n \sum_{r=0}^{k-1} m^r < cm^k$ for some uniform constant $c$. 

To bound the Hausdorff dimension of $C$ from above, we consider for each $k$ the cover of $C$ given by taking the set of intervals that are the connected components of $C_k$. Denote the lengths of these intervals by $\ell_{k,1},\dots,\ell_{k,N_k}$, where $N_k < cm^k$ is their number. Observe that the mesh of this cover goes to 0 as $k$ goes to infinity, since $\ell_{k,j} \leq 2 \max_{i\in \mathcal{A}} |J^{(k)}_i|$.

If $\mu = \frac{\min_{i\in \mathcal{A}} |I_i|}{\max_{i\in \mathcal{A}} |I_i|}$, then for any $p\in \mathcal{P}_{k-1}$, $|G_p| \geq m^{-1} \mu |\mathcal{J}(p)|$, which means that at every step the total length of gaps that we remove has proportion at least $m^{-1}\mu$ of the length of $C_k$ and hence by induction, for all $k$, 
\[\sum_{i=1}^{N_k} \ell_{k,i} = |C_k| \leq (1-m^{-1}\mu)^k.\]

Now take $\beta < 1$. Then by the power means inequality,
\[\left( \frac{\sum_{i=1}^{N_k} \ell_{k,i}^\beta}{N_k} \right)^{\frac{1}{\beta}} \leq \frac{\sum_{i=1}^{N_k} \ell_{k,i}}{N_k} = \frac{|C_k|}{N_k},\]
and hence
\[\sum_{i=1}^{N_k} \ell_{k,i}^\beta \leq |C_k|^{\beta} N_k^{1-\beta} \leq c \left((1-m^{-1}\mu)^\beta m^{1-\beta}\right)^k.\]

From this we deduce that if $\beta_0 < 1$ is such that $(1-m^{-1}\mu)^{\beta_0} m^{1-\beta_0} = 1$, then for $\beta > \beta_0$, the $\beta$-dimensional Hausdorff measure of $C$ is 0, and hence we deduce that the Hausdorff dimension of $C$ is at most $\beta_0$.
\end{proof}

\begin{cor}\label{cor:hdim}
Suppose that $(M,\omega)$, $\gamma$ satisfy the conditions of the statement of \Cref{thm:invariant}, with $\omega$ $g_t$-periodic. Then for every $\tilde x \in \tilde M_\gamma$, the closure of the orbit of $\tilde x$ under the vertical flow has Hausdorff dimension strictly less than 2.
\end{cor}

\begin{proof}
By assumption, $G_t^\mathcal{V}$ has some negative Lyapunov exponents. Hence there exists a class $\alpha_1 \in K \otimes \R$ and a $\lambda < 0$, such that $G^\mathcal{V}_{t_0} \alpha_1 = e^\lambda \alpha_1$, where $t_0$ is the length of the closed Teichmüller geodesic containing $(M,\omega)$, i.e. such that $g_{t_0}(M,\omega) = (M,\omega)$. 

Let $T:I\to I$ be a Poincaré section for the vertical flow on $M$ as before. Then for some $N$ that corresponds to a multiple of the time $t_0$, we have $\mathcal{R}^N(T) = T$, and we can assume, by taking a multiple of $t_0$ and $N$ if necessary, that the corresponding matrix $A$ is positive.

Then if $\psi_1$ is the piecewise constant function on $I$ defined by $\psi_1(x) = \langle \alpha_1, \xi(x) \rangle$, $\psi_1$ when seen as a vector is an eigenvector of $A^T$ with eigenvalue $e^\lambda <1$. Let $h_1: I \to \R$ be the transfer function corresponding to $\psi_1$. 
By \Cref{thm:hdim} we know that for any $z\in \R$, $h_1^{-1}(z)$ has Hausdorff dimension strictly less than 1. 

Extend $\{\alpha_1\}$ to a basis $\{\alpha_1,\dots,\alpha_{d_+}\}$ for the stable subspace of $G_t^\mathcal{V}$, and define $h_i$ as in the proof of \Cref{thm:invariant}. In the notation of that theorem, 
\[\tilde h(x,a)= \left(h_1(x) - \sum_{j=1}^m b_{1j}a_j , \dots, h_{d_+}(x) - \sum_{j=1}^m b_{d_+ j}a_j\right)\]
is invariant by $T_\phi$ up to an additive factor of $\Lambda = C \cdot \Z^{d_+}$.

Consider a fixed $a \in \Z^m$ and a fixed $z = (z_1,\dots,z_{d_+}) \in \R^{d_+}$. Let us show that 
\[\dim_H \left((I\times \{a\}) \cap \tilde h^{-1}(z+\Lambda)\right) < 1.\]

Indeed, as each $h_i$ is bounded, $(I\times \{a\}) \cap \tilde h^{-1}(z+\Lambda)$ is the finite union
\[(I\times \{a\}) \cap \tilde h^{-1}(z+\Lambda) = (I\times \{a\}) \cap  \bigcup_{g\in G}\tilde h^{-1}(z+g),\]
for some finite $G \subset \Lambda$.
As the Hausdorff dimension of a finite union of sets is the maximum of the individual dimensions, it is enough to show that for any $z$, 
\[\dim_H \left((I\times \{a\}) \cap \tilde h^{-1}(z)\right) < 1.\]

Now since our $a$ is fixed, $\tilde h(x,a) = z \iff (h_1(x),\dots,h_{d_+}(x)) = (z'_1,\dots,z'_{d_+})$, for $z'_i = z+\sum_{j=1}^m b_{ij}a_j$. 
Hence in particular,
\[(I\times \{a\}) \cap \tilde h^{-1}(z) \subset h_1^{-1}(z'_1) \times \{a\},\]
and by \Cref{thm:hdim}, the latter set has Hausdorff dimension strictly less than 1.

Thus we conclude by \Cref{thm:invariant} that for any $(x,a) \in I \times \Z^m$,
\[\dim_H (\overline{ \{T_\phi^n(x,a): n \in \Z \}} ) < 1,\]
and hence for any $\tilde x \in \tilde M_\gamma$,
\[\dim_H (\overline{ \{\varphi^v_t(\tilde x): t \in \R \}} ) < 2.\]
\end{proof}

\subsection{Application to the wind-tree model}
\begin{proof}[Proof of \Cref{thm:main}]
Recalling \Cref{sec:wind-tree}, we see that for any parameters $a,b$, and any $\theta$ such that $r_{\pi/2-\theta}X(a,b)$ is $g_t$-periodic and at least one of $\gamma_h$ or $\gamma_v$ is unstable, $H_1(X,\Z)$ has an orthogonal splitting satifying the conditions of \Cref{thm:invariant}. Indeed, without loss of generality suppose that $\gamma_h$ is unstable. Then if $K = E^{-+}$ and $K^\perp = E^{++} \oplus E^{+-} \oplus E^{--}$, then this is an orthogonal splitting. By the assumption that $\gamma_h$ is unstable, $K$ has one positive and one negative Lyapunov exponent. Hence applying \Cref{cor:hdim} for the $\Z$-cover of $X(a,b)$ given by $\gamma_h$, we deduce that the closure of any orbit of $(\varphi^\theta_t)_{t\in\R}$ on $\tilde X(a,b)_{\gamma_h}$ has Hausdorff dimension less than 2.
Since $X_\infty(a,b)$ is a $\Z$-cover of $\tilde X(a,b)_{\gamma_h}$, orbit closures of  $(\varphi^\theta_t)_{t\in\R}$ on $X_\infty(a,b)$ are contained within the lifts of orbit closures on $\tilde X(a,b)_{\gamma_h}$, and hence also have Hausdorff dimension less than 2. Since $X_\infty(a,b)$ is the unfolding of the wind-tree model $W(a,b)$, we obtain the desired result.
\end{proof}

\appendix
\section{Plotting invariant sets}\label{sec:plotting}
Here we briefly describe some of the steps of plotting examples of the invariant sets given by \Cref{thm:invariant}, as in \Cref{fig:InvSet}. We do this in the $g_t$-periodic case, where the computation can be done more rigorously. \\

The code used to generate the figure, using the \texttt{surface-dynamics} library in Sage (\cite{surface-dynamics}) is contained in the notebook \texttt{Invariant\_sets\_for\_periodic\_type\_windtrees.ipynb}, which is available on ArXiv, as well as on the website
\begin{center}
\url{https://yuriytumarkin.github.io/windtree/}
\end{center}

The main steps are:
\begin{enumerate}
\setcounter{enumi}{-1}
\item Find parameters $a,b,\theta$ for which $r_{\pi/2-\theta} X(a,b)$ is $g_t$-periodic.
\item Find the stable vectors $\alpha_i$ and compute the transfer functions $h_i$.
\item Find the level sets of $h_i$.
\end{enumerate}

\subsubsection*{0. Finding periodic parameters}

One method to find periodic parameters is to look for loops in the Rauzy graph (see \cite{SinaiUlcigrai}). The Rauzy graph for IETs corresponding to the flow on $X$ is too large for this to be effective, but one can look for loops in the Rauzy graph corresponding to the flow on the genus 2 surface $Y$. If one find parameters $a,b,\theta$ such that $r_{\pi/2-\theta}Y(a,b)$ is $g_t$-periodic, then the fourfold cover $r_{\pi/2-\theta}X(a,b)$ is also $g_t$-periodic, with possibly a longer period.

\subsubsection*{1. Finding stable vectors and computing the transfer function}
In the periodic case, taking the matrix $A = A^{(0,N)}$ corresponding to the period $N$, the stable vectors can be found simply as the eigenvectors of $A^T$ with eigenvalues less than 1, in the relevant blocks $E^{+-}$ and $E^{-+}$. Knowing the stable vectors $\alpha_i$ we can find the corresponding cocycles $\psi_i$, which allows one to compute the transfer functions $h_i$ on the orbit of 0 using the cohomological equation $h(T^k(0)) = h(0) +s_k \psi (0)$. However this isn't enough to determine whether $h_i$ takes a certain value on a given interval or not.

\subsubsection*{2. Finding level sets}
In the periodic case, one can explicitly compute the values of the function $f$ defined in \Cref{sec:h_formula} and then use the expression with geometric tails as derived in \Cref{prop:h_formula} to get precise upper and lower bounds for $h$ on each interval $I_i$. Let $\mathcal{P}^{(1)}$ be the dynamical partition of $I$ into intervals, where each interval is a floor of some tower $Z_j, j\in \mathcal{A}$. Using the self-similarity of $h$, this allows one to compute upper and lower bounds for $h$ on each interval belonging to $\mathcal{P}^{(1)}$, if one can compute the value of $h$ at its left endpoint. To do this, one can use the following observation:

\begin{lemma}\label{lem:tau}
Let $T$ be a periodic type IET, with $\mathcal{R}^N(T) = T$. Let $A = A^{(0,N)}$ be the corresponding Rauzy-Veech matrix. Suppose $\psi$ is an eigenvector of $A^T$ with eigenvalue $\lambda <1$, and let $h$ be the corresponding continuous transfer function. Let $0=x_0 < x_1 < \dots <x_n = 1$ be the discontinuities of $T$. Then if the vector $\tau \in \R^n$ is defined by 
\[\tau_i = h(x_i) - h(x_{i-1}),\]
it satisfies $A\tau = \lambda^{-1} \tau.$
\end{lemma}
\begin{proof}
If the discontinuities of the induced map $T_1$ are $0 = x_0^{(1)} < x_1^{(1)} < \dots < x_n^{(1)}$, then by self-similarity the vector $\tau^{(1)}$ given by $\tau^{(1)}_i = h(x^{(1)}_{i}) - h(x^{(1)}_{i-1})$ is simply $\tau^{(1)} = \lambda \tau$. Indeed, if $(k_r)_{r\in \N}$ is an increasing sequence of integers such that $T^{k_r}(0) \to x_i$, then $T_1^{k_r}(0) \to x^{(1)}_i$, and applying \Cref{lem:induced_h} we see that $h(x^{(1)}_i) = \lambda h(x_i)$.

Since for each $j$, the interval $I_i$ consists of $A_{ij}$ floors of the tower $Z_j$, and the difference between the value of $h$ at the left and right endpoints of each floor is $\tau^{(1)}_j$, it follows that $\tau_i = \sum_{j\in \mathcal{A}} A_{ij} \tau^{(1)}_j$, and thus $\tau = A \tau^{(1)} = \lambda A\tau.$
\end{proof}

This may not be enough by itself to identify $\tau$, since the $\lambda^{-1}$-eigenspace of $A$ may have higher dimension, but one can also use additional considerations such as the values of $h$ at $T^{-1}(0)$ and $T^{-1}(1)$. Thus for example, if $\beta  = \pi_b^{-1}(1)$, so that $T^{-1}(0)$ is the left endpoint of $I_\beta$, then
\[-\psi_\beta = h(T^{-1}(0)) = \sum_{\alpha \in \mathcal{A}: \pi_t(\alpha) < \pi_t(\beta)} \tau_\beta.\]
Combining this with \Cref{lem:tau} one can identify the vectors $\tau$. Knowing also the substitution corresponding to $A$, for any interval $K \in \mathcal{P}^{(1)}$, one can locate $K$ inside $I$, compute the value of $h$ at the left endpoint of $K$, and then use the upper and lower bounds to check whether $h$ takes a certain value on $K$ or not.

\bibliography{nonergodicity}{}
\bibliographystyle{amsalpha}

\end{document}